\documentclass
{amsart}
\usepackage{soul,color}
\usepackage{fullpage}
\usepackage[all]{xy}
\usepackage{color}
\usepackage{amssymb}
\usepackage{amsfonts}
\usepackage{MnSymbol}
\makeindex
\theoremstyle{plain}
\newtheorem{theorem}{Theorem}[section]
\newtheorem{corollary}[theorem]{Corollary}
\newtheorem{proposition}[theorem]{Proposition}
\newtheorem{lemma}[theorem]{Lemma}
\newtheorem{question}[theorem]{Question}

\theoremstyle{definition}
\newtheorem{definition}[theorem]{Definition}
\newtheorem{remark}[theorem]{Remark}
\newtheorem{fact}[theorem]{Fact}
\newtheorem{example}[theorem]{Example}
\newtheorem{problem}[theorem]{Problem}
\newcommand{\HG}{\mathbb{H}(\omega)}

\newcommand{\Mt}{M_t}

\begin{document}

\title{Generalized Heisenberg Groups and self-duality}

\author[Bonatto]{Marco Bonatto}
\address[M. Bonatto]{\hfill\break
Charles University
\hfill\break
KA MFF UK, Sokolovsk\'{a} 83, 18675 Praha 8 
\hfill\break
Czech Republic}
\email{marco.bonatto.87@gmail.com} 

\author[Dikranjan]{Dikran Dikranjan}\address[D. Dikranjan]{\hfill\break
Dipartimento di Matematica e Informatica
\hfill\break
Universit\`{a} di Udine
\hfill\break
Via delle Scienze  206, 33100 Udine
\hfill\break
Italy}
\email{dikranja@dimi.uniud.it}



\begin{abstract}
This paper compares two generalizations of Heisenberg groups and studies their connection to one of the major open problems in the field of locally compact abelian groups, namely the description of the self-dual locally compact abelian groups (\cite{H0}, \cite{H}). 
The first generalization is presented by the so called {\em generalized Heisenberg groups} $\HG$, defined in analogy with the classical Heisenberg group, and the second one is inspired by the construction proposed by Mumford in \cite{Mu} and named after him as {\em Mumford groups}.\\
These two families can be defined also in the framework of topological groups. We investigate the relationship between locally compact Mumford groups, locally compact generalized Heisenberg groups with center isomorphic to $\mathbb{T}$, and symplectic self-dualities.
\end{abstract}

\maketitle
\section{Introduction}

\subsection{The Heisenberg group and its generalizations}

For a commutative unitary ring $R$ one can define  the linear group of all unitriangular matrices:
$$
\mathbb H_R:= \left (\begin{array}{ccc} 1 & b & a\\
0 & 1 & c\\
0 & 0 & 1\\
 \end{array}\right)
 $$
where $a,b,c \in R$. When $R= \mathbb R$ is the ring of reals, one obtains the classical real 3-dimensional Heisenberg group $\mathbb{H}_{\mathbb R}$. 
This group and its higher-dimensional versions are important in analysis and quantum mechanics. The name ``Heisenberg group" is motivated by the fact that the Lie algebra of $\mathbb{H}_\mathbb R$ gives the Heisenberg commutation relations in quantum mechanics. Even the smallest commutative unitary ring $\mathbb Z_2$ provides a quite interesting Heisenberg group, as $\mathbb{H}_{\mathbb Z_2}$ is the dihedral group $D_4$. 

Here we study two generalizations of Heisenberg groups and their connection to self-dual locally compact abelian groups (this issue will be discussed with more details in \S \ref{SD}). 

The first generalization due to Megrelishvili is presented by the so called {\em generalized Heisenberg groups} $\HG$, defined in analogy with $\mathbb H_R$. The elements of $\HG$ can still be thought of as being matrixes as above, with entries taken from three arbitrarily chosen abelian groups $E, F$ and $A$, so that $b \in E$, $a\in A$ and $c\in F$ and the multiplication follows the usual rule of matrix multiplication. Accordingly, in the product 
$$
\left (\begin{array}{ccc} 1 & b & a\\
0 & 1 & c\\
0 & 0 & 1\\
 \end{array}\right)\left (\begin{array}{ccc} 1 & b' & a'\\
0 & 1 & c'\\
0 & 0 & 1\\
 \end{array}\right)
$$
the only non-trivial ``product" $b\cdot c':=\omega(b,c')$ is given by a bilinear map $\omega: E\times F\longrightarrow A$ assigned in advance (see Definitions \ref{LinMap} and \ref{definition generalized H}). This operation makes 
$\HG$ a nilpotent group of class 2, a central extension with $Z(\HG)\cong A$ and $\HG/Z(\HG) \cong E\times F$. When $E, {F}$ and ${A}$ are topological abelian groups and $\omega$ is continuous,  $\HG$ is a topological group (when equipped with the product topology), largely used in view of the flexibility in the choice of the parameters $E, F, A$ and $\omega$. 
In particular, these groups were prominently used in the theory of minimal groups for building relevant examples and counter-examples 
in the non-abelian case because of their simple algebraic structure  (\cite{DM,DM1,DS,Me1,Me2,Mi,Re,Shlo}). 

A relevant example is obtained by taking a LCA group $E$, its Pontryagin dual $F = \widehat{E}$ and the standard evaluation bilinear map $\omega: E\times \widehat{E}\longrightarrow \mathbb T$. These Heisenberg groups, denoted by $\mathbb H(E)$ and called  {\em Mackey-Weil groups} (after \cite{Weil}), are inspired by the classical 
\emph{Weyl-Heisenberg group} $\mathbb H(\mathbb R^n)$. 

Another route for generalization of Heisenberg groups was proposed by Mumford (\cite{Mu}), who generalized the notion of a Mackey - Weil group as follows. For every nilpotent group $G$ of class 2 consider the factorization $B : G/Z(G) \times G/Z(G) \to Z(G)$ of the commutator map $G\times G \to Z(G)$ through the Cartesian square of the canonical map $G \to G/Z(G)$. This gives rise to a canonical map:
 \begin{equation}\label{Intro}
M: G/Z(G)\longrightarrow Hom(G/Z(G),Z(G)), \quad M(xZ(G))(yZ(G)):=B(xZ(G), yZ(G)) = [x,y]
\end{equation}
that  we call {\it Mumford map} (in other words,  $M$ is the evaluation map of $B$, see also Fact \ref{charact of N_2} and Definition \ref{Mumford map}).
In case $G$ carries a locally compact group topology and $Z(G)\cong \mathbb T$, the target group in (\ref{Intro}) is nothing else but the Pontryagin dual of $G/Z(G)$ (see Remark \ref{topological commutator} for more detail).
In order to describe the properties of such a group $G$ related to its irreducible unitary representations, Mumford \cite{Mu} imposed additionally the condition that $M$ is a topological isomorphism (these groups were used also in \cite{Prasad,Vemuri}). Clearly, every Mackey - Weil group has this property.

Here we  adopt a more general setting, by removing any restraint on $Z(G)$. We only keep the condition {on the Mumford map} to be a topological isomorphism and we call these groups {\em Mumford groups} (Definition \ref{M-groups}). Clearly, the locally compact Mumford groups $G$ with 
center  topologically isomorphic to $\mathbb T$ are precisely the above mentioned groups isolated in \cite{Mu,Prasad,Vemuri}. We give criteria for a generalized Heisenberg group to be a Mumford group, and we give criteria for a Mumford group to be a generalized Heisenberg group. It turns out that the intersections of these two classes, in the case of locally compact groups with center $\mathbb T$, is precisely the class of Mackey - Weil groups (Theorem \ref{gen iff sss}).

\subsection{Self-dual LCA groups}\label{SD}

A locally compact abelian group $L$ is said to be self-dual if there exists a topological group isomorphism
\begin{equation}\label{SSSSD}
\nabla: L \to \widehat{L},
\end{equation}
called  {\em self-duality} of $L$, in such a case the group $L$ is called {\em self-dual}. Many locally compact abelian groups  are known to be self-dual, examples include (but are not limited to) all finite groups, the reals $\mathbb R$, the fields $\mathbb Q_p$ of $p$-adic numbers (actually, the underlying group of any locally compact field), etc. 

Nevertheless, the following question (this is Question 505 in  \cite{WC}, previously mentioned also in \cite{HR}, \cite{A})
remains one of the most prominent open problems in the area of dualities and locally compact abelian groups: 

\begin{problem} \cite[Question 3A.1]{WC}	Classify the self-dual locally compact abelian groups.	
\end{problem}

One can find an extended literature on this problem. The ``classical line" of study of the self-dual locally compact abelian groups imposes 
 eventual additional restraints on the structure of the groups (\cite{H0,H,RS,RS1,Wu}). A somewhat more recent line  
 prefers to avoid imposing structural restraints on the groups involved, but makes heavy use of homological algebra \cite{FH1,FH2}, or
uses essentially properties on the bilinear form $b_\nabla: L\times L \longrightarrow \mathbb{T}$ associated to the  self-duality $(L,\nabla)$ (\cite{Prasad,Vemuri}), defined as follows: $b_\nabla(x,y) = \nabla(x)(y)$
for $x,y \in L$. In the opposite direction, self-dualities are often built via appropriate bilinear forms $\omega: L\times L \longrightarrow \mathbb{T}$.
 In the case of a locally compact field $L$, this is simply the field multiplication in $L$ and this form $\omega$ 
 is symmetric (Definition \ref{LinMap}(a$_3$)), although in other cases the bilinear form of a self-duality can be {\em alternating} (Definition \ref{LinMap}(a$_2$)), and in such a case one speaks of {\em symplectic self-dualities} (\cite{Prasad}). 
 
 The connection between locally compact Mumford groups $G$ with $Z(G) \cong \mathbb T$ and self-dualities is clear comparing (\ref{Intro}) (with $Z(G) \cong \mathbb T$) and (\ref{SSSSD}) (with $ L = G/Z(G)$), as pointed out already in \cite{Prasad}. For more details about the relation between generalized Heisenberg groups, Mackey - Weil groups and Mumford groups, on one hand, and  symplectic self-dualities on the other hand, see Theorem \ref{gen iff sss}. \\

The first named author would like to thank Prof. Libor Barto and Prof. David Stanovsk\'y for the support provided during his visiting research period at the Algebra department of the Charles University in Prague.

The second named author would like to sincerely thank Maria Jesus Chasco and Sergio Ardanza for pointing out to him the relevant paper \cite{Prasad}, as well as numerous enlightening discussions on self-duality during his visit in Pamplona in October 2011. This paper would not appear 
without the longs years of fruitful collaboration with Michael Megrelishvili allowing the first named author to realize the potential power of generalized Heisenberg groups as a formidable supply of relevant examples in topological group theory. 

\section{Background from group theory}

Let us recall first some well-known notions used in the sequel. As usual, for  a group  $G$ the map:
\[[\cdot,\cdot]:G\times G\rightarrow G, \quad (x,y)\mapsto [x,y]=x^{-1}y^{-1} xy\]
is called \textit{the commutator map} and $[x,y]$ is called \textit{the commutator of x and y}. The subgroup:
\[[G,G]=\langle [x,y],\ x,y\in G\rangle \]
is called \textit{the commutator subgroup}.

For any group $G$  the commutator map $G\times G \to  [G,G]$ can be factorized through the canonical map $\pi :  G \to  G/Z(G) $. More precisely, 
 the map $B$ obtained by setting 
\begin{equation}\label{B}
B:G/Z(G)\times G/Z(G)\longrightarrow [G,G],\quad (xZ(G),yZ(G))\mapsto [x,y]
\end{equation}
is well defined and the following diagram is commutative:
\medskip

\begin{equation}\label{diagNEW}
\xymatrixcolsep{5pc}\xymatrix{
G\times G\ar[r]^-{\pi\times\pi}  \ar[dr]^{[\cdot,\cdot]} & G/Z(G) \times G/Z(G) \ar[d]^{B}\\
& [G,G] &
}
\end{equation}

\subsection{Bilinear Maps}
\begin{definition}\label{LinMap}
Let $E,F$ and $A$ be Abelian groups. A map $\omega: E\times F\longrightarrow A$ is said to be {\em bilinear} if
\begin{displaymath}
\omega(x+x',y)=\omega(x,y)+\omega(x',y)\  \ \mbox{ and }\  \ \omega(x,y+y')=\omega(x,y)+\omega(x,y')
\end{displaymath} for every $x,x'\in E$ and every $y,y'\in F$. 

\begin{itemize}
\item[(a)] A bilinear map $\omega$ is said to be:
\begin{itemize}
\item[(a$_1$)] {\em separated} if for every $0\neq x\in E$ there exists $y\in F$ such that $\omega(x,y)\neq 0$ and for every $0\neq y'\in F$ there exists $x'\in E$ such that $\omega(x',y')\neq 0$;

\item[(a$_2$)]  if $E=F$, $\omega$ is {\em alternating} if $\omega(x,x)=0$ for every $x\in E$;
\item[(a$_3$)] if $E=F$, $\omega$ is {\em symmetric} if $\omega(x,y)=\omega(y,x)$ for all $x,y\in E$.
\end{itemize}
\item[(b)] If $E=F$, a subgroup $H$ of $E$ is said to be {\em isotropic with respect to} $\omega$, if $\omega (H,H) = 0$ (i.e., $\omega(x,y)=0$ for every $x,y\in H$). 
\item[(c)] Let $E_1$ be an abelian group  and let $\omega_1: E_1\times E_1\longrightarrow A$ be a bilinear map. 
We say that $\omega$ and $\omega_1$ are {\em isomorphic}, if there exists an isomorphism $\xi: E \to E_1$, such that $\omega_1 \circ (\xi\times \xi) = \omega$.  
\end{itemize}

\end{definition}

\begin{remark} \label{omegaE}
Let $E,F$ and $A$ be Abelian groups and let $\omega$ be a map $\omega: E\times F\longrightarrow A$. Define the maps $\omega_E$ and $\omega_F$ by setting:
\begin{eqnarray}
\omega_E:E\longrightarrow A^F , &\quad &\omega_E(x)(y)=\omega(x,y)\label{B tilde1}\\
\omega_F:F\longrightarrow E^F,&\quad &\omega_F(y)(x)=\omega(x,y)\label{B tilde2}
\end{eqnarray}
for every $x\in E$, $y\in F$. The properties of the map $\omega$ depend on the properties of the maps \eqref{B tilde1} and \eqref{B tilde2}: 
\begin{itemize}
\item[(a)] $\omega$ is bilinear if and only if $\omega_E$ and $\omega_F$ are group homomorphisms which take values in $Hom(F,A)$ and $Hom(E,A)$, respectively.
\item[(b)] $\omega$ is separated if and only if both $\omega_E$ and $\omega_F$ are injective and in this case $E$ embeds in $Hom(F,A)$ and $F$ embeds in $Hom(E,A)$ (and $Im(\omega_E)$ separates the points of $F$ and $Im(\omega_F)$ separates the points of $E$).

\end{itemize}
\end{remark}

\begin{remark}
\begin{itemize}
\item[(a)] Obviously, $\omega: E\times E\longrightarrow A$ is alternating if and only if every cyclic subgroup of $E$ is isotropic with respect to $\omega$.
\item[(b)] Leading examples of symmetric bilinear forms come from the ring multiplication $\omega: R\times R\longrightarrow R$ in a commutative ring $R$ 
(i.e., $\omega(x,y) = xy$, for $x,y \in R$). A relevant example of an alternating bilinear form of a different nature can be found in Example \ref{Exa:ssd}. 
\end{itemize} 
\end{remark}

\subsection{The class $\mathcal{N}_2$ and the Mumford map}\index{$\mathcal{N}_2$}

A group $G$ is said to be a {\em nilpotent group of class 2}, if $G/Z(G)$ is abelian, i.e., $[G,G]\leq Z(G)$. This class of groups will be denoted by $\mathcal{N}_2$. For the sake of convenience we collect here several easy equivalent properties which distinguish these groups:   

\begin{fact}\label{charact of N_2}  Let $G$ be a group, then the following are equivalent: 
\begin{enumerate}
 \item $G\in \mathcal{N}_2$, i.e., $G$ is a nilpotent group of nilpotency class $2$;
 \item $[G,G]\leq Z(G)$;
 \item the factorized commutator map $B$ defined by setting
\begin{displaymath}
B: G/Z(G)\times G/Z(G)\longrightarrow G, \quad (xZ(G),yZ(G))\mapsto [x,y]
\end{displaymath} 
is an alternating bilinear and separated map.
\end{enumerate}
\end{fact}


For a group $G\in \mathcal{N}_2$, the group $G/Z(G)$ is abelian and the factorized commutator map $B$ is bilinear (see Fact \ref{charact of N_2}). 
The related map $B_{G/Z(G)}$, as in Remark \ref{omegaE}, will be largely used in the paper, so we define it as follows.
\begin{definition}\label{Mumford map} Let $G\in \mathcal{N}_2$. The map $M$ defined by setting
\begin{equation}\label{mumford map formula}
M: G/Z(G)\longrightarrow Hom(G/Z(G),Z(G)), \quad M(xZ(G))(yZ(G)):=B(xZ(G), yZ(G)) = [x,y]
\end{equation}
for every $x,y\in G$, will be called the {\it Mumford map} of $G$.
\end{definition}
 The map $M$ characterizes the central extension of a given Abelian group $A$ up to isomorphism \cite[Proposition 3.5]{Prasad} (this proposition can be extended to arbitrary central extensions).


\begin{lemma}\label{commutator map} For a group 
 $G\in \mathcal{N}_2$ let $K=G/Z(G)$. Then the Mumford map $M:K\longrightarrow Hom(K,Z(G))$ is an injective homomorphism. In particular $K$ embeds in $Z(G)^K$.
\end{lemma}

\begin{proof}  The factorized map $B$ is bilinear and separated, by Fact \ref{charact of N_2}. According to Remark \ref{omegaE}, the map $M$ is an injective homomorphism and therefore $K$ embeds in $Z(G)^K$.
\end{proof}


\begin{lemma}\label{isotropic and abelian}
Let $G\in \mathcal{N}_2$ and let $\pi: G	\longrightarrow K=G/Z(G)$ be the canonical projection. Then $L \mapsto \pi^{-1}(L)$ defines a one-to-one 
monotone correspondence between isotropic subgroups of $K$ with respect ${B}$ and abelian subgroups of $G$ containing $Z(G)$
\end{lemma}
\begin{proof} It is enough to note that a subgroup $H\leq G$ is abelian if and only if $[H,H]= \{e\}$, while $L \leq K$ is isotropic if and only if $B(L,L) = 0$. 
\end{proof}
Since maximality is preserved under this monotone correspondence, maximal isotropic subgroups of $K$ correspond to maximal abelian subgroups of $G$ containing $Z(G)$. 
	
\section{Generalized Heisenberg Groups}

In the class $\mathcal{N}_2$ there is a subclass of groups which is given by the following construction proposed in \cite{Me1}.

\begin{definition}\label{definition generalized H}
\label{heisenberg} Let $E,F,A$ be Abelian groups, $\omega$ a separated bilinear map from $E\times F$ to $A$. The group $\HG=(E\times F\times A,\cdot)$, defined by setting
\begin{equation*}
(x,y,z)\cdot (x',y',z')=(x+x',y+y',z+z'+\omega(x,y'))
\end{equation*}
for every $x,x'\in E, y,y'\in F$ and $z,z'\in A$ is called \textit{generalized Heisenberg group} and denoted by $\mathbb{H}(E,F,A,\omega)$ or simply by $\HG$.
\end{definition}

\begin{remark}\label{B and tilde for generalized}
\begin{itemize}
\item[(a)] Generalized Heisenberg groups are nilpotent groups of class 2. Indeed, $Z(\HG)=\lbrace 0\rbrace\times\lbrace 0	\rbrace\times A$ and $\HG/Z(\HG) \simeq E\times F$ is abelian. 

To compute the factorized commutator map $B$ we note that 
\begin{eqnarray*}
B((x,y),(x',y'))&=&M(x,y)(x',y') = [(x,y,0),(x',y',0)]\\
&=&\omega(x,y')-\omega(x',y).
\end{eqnarray*}
Therefore, 
\[B:(E\times F)\times (E\times F)\longrightarrow A, \quad ((x,y),(x',y'))\mapsto \omega(x,y')-\omega(x',y).\]
Hence, the commutator subgroup of $\HG$ can be computed as follows
$$
[\HG,\HG] = \{0\}\times\{0\}\times \langle \omega(x,y), \ (x,y)\in E\times F\rangle.
$$

\item[(b)] The following alternative description of $\HG$ is possible for the reader familiar with cohomology \cite{Mac}. 
 Let $\sigma: E\times F=\HG/Z(\HG) \to \HG$ be the obvious section defined by $\sigma(x,y) = (x,y,0)$ for $(x,y)\in E\times F$. Then, with respect to this section,  the group $\HG$, as a central extension of $A$, is given by the bilinear group cocycle obtained by setting
\begin{displaymath}
\gamma: (E\times F)\times (E\times F)\longrightarrow A,\quad  ((x,y),(x',y'))\mapsto \omega(x,y'),
\end{displaymath}
as $\gamma((x,y), (x',y')) = \sigma(x,y)\sigma(x',y')\sigma(x+x',y+y')^{-1} = (0,0,\omega(x,y'))$.
If $i :E\longrightarrow\HG$ and $j: F\longrightarrow\HG$ are the canonical injections, then
\[\omega:E\times F\longrightarrow A, \quad (x,y)\mapsto \omega(x,y)=[i(x),j(y)].\]
To make our exposition free of homological background, we are not going to pursue this line in the sequel (except in Remark \ref{last:Rem}). 
\end{itemize}
\end{remark}

\subsection{A characterization of generalized Heisenberg groups}
A generalized Heisenberg group $\HG$ can be characterized in terms of a specific property of  its subgroup lattice which is stated in the next lemma. 

\begin{lemma}\label{heisenberg as semidirect}
Let $G=\mathbb{H}(E,F,A,\omega)$ be a generalized Heisenberg group, where $E$ and $F$ are  identified with subgroups of $G$ in the obvious way. Then
\begin{itemize}
\item[(i)] $M_1=  Z(G)\times E $ and $M_2 = Z(G)\times F$ are maximal Abelian subgroups of $\HG$;
\item[(ii)] $G=M_1\cdot F=M_2\cdot E=E\cdot F\cdot Z(G)$;
\item[(iii)] $M_1\cap F=M_2\cap E=\lbrace 1 \rbrace$ and $M_1\cap M_2=Z(G)$. 
\end{itemize}
In particular, $M_1$ and $M_2$ are normal subgroups of $G$ and 
\begin{equation}\label{reticolo}
G\simeq M_1\rtimes  F\simeq M_2\rtimes  E. 
\end{equation}
\end{lemma}

Before the proof we summarize in the diagram below how the relevant subgroups from  (i)--(iii) are placed in the subgroups lattice of $\HG$: 

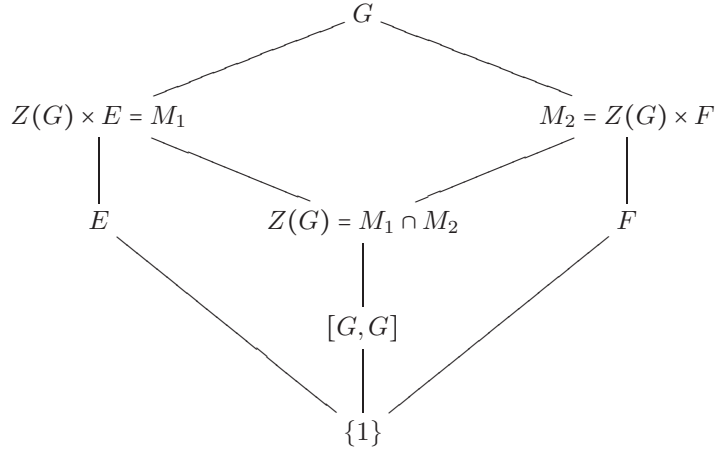
\begin{center}\label{structure of H}
$\xymatrix{
& G \ar@{-}[dr]   \\
 Z(G)\times E= M_1 \ar@{-}[ur] \ar@{-}[dr] \ar@{-}[d]& & M_2= Z(G)\times F \ar@{-}[dl] \ar@{-}[d]\\
E \ar@{-}[ddr]& Z(G)=M_1\cap M_2 \ar@{-}[d]& F \ar@{-}[ddl]\\
&[G,G]\ar@{-}[d]&\\
&\lbrace 1\rbrace& 
\\ 
}
$
\end{center} 

\begin{proof} The proof of (i) easily follows from the fact that the bilinear form $\omega$ is separated, while (ii) and (iii) follow from the definition of $\HG$. 

The subgroups $M_1$ and $M_2$ are normal as they contain the subgroup $[G,G]$. Now (\ref{reticolo}) follows from (ii) and (iii). 
\end{proof}

Now we show that the properties from the previous lemma completely characterize generalized Heisenberg groups. 

\begin{theorem}\label{caratt generalized} A group $G\in\mathcal{N}_2$ is a generalized Heisenberg group if and only if there exist  maximal Abelian subgroups  $M_1, M_2$ of $G$
such that:
\begin{enumerate}
\item $M_1\cap M_2=Z(G)$;
\item $G=M_1 \cdot M_2$;
\item $Z(G)$ splits as a direct summand in $M_1$ and $M_2$.
\end{enumerate}
\end{theorem}

\begin{proof} The necessity was already established in Lemma \ref{heisenberg as semidirect}. 
 
To prove the sufficiency assume that $G$ is a group satisfying (1)--(3). 
Then exist subgroups $E$ and $F$ of $M_1$ and $M_2$ respectively, such that $M_1\simeq E\times Z(G)$ and $M_2\simeq F\times Z(G)$. Let $i:E\longrightarrow G$ and $j:F\longrightarrow G$ be the embeddings of $E$ and $F$ in $G$. The map given by: 
 \begin{displaymath}
\omega: E\times F\longrightarrow Z(G),\quad (x,y)\longrightarrow [i(x),j(y)]
\end{displaymath}
is a separated bilinear map, and the group $\HG=\mathbb{H}(E, F, Z(G),\omega)$ defined as in \ref{definition generalized H} is a generalized Heisenberg group. Note that every element $g\in G$ can be written in a unique way as $g=j(y)i(x)z$ for some $x\in E,y\in F$ and $z\in Z(G)$. So, the map $\phi: \HG\longrightarrow G$ defined by setting $\phi(x,y,z)= j(y)i(x)z$ for every $x\in E$, $y\in F$ and $z\in Z(G)$ is a bijection. Since $G\in\mathcal{N}_2$, all commutators are central elements, therefore we have
\begin{eqnarray*}
\phi((x,y,z)(x',y',z'))&=&\phi(x+x',y+y',z+z'+\omega(x,y'))=j(y+y')i(x+x') z z'[i(x),j(y')]\\
&=&j(y)j(y')i(x)[i(x),j(y')]i(x') z z'=\\
&=& j(y)i(x) z j(y')i(x') z'=\\
&=&\phi(x,y,z)\phi(x',y',z')
\end{eqnarray*} 
for every $x,x'\in E, y,y'\in F$ and $z,z' \in Z(G)$. This proves that $\phi$ is a group isomorphism.
\end{proof}


\subsection{Standard Heisenberg group}
Let us introduce now a family of specific generalized Heisenberg groups.
\begin{lemma}\label{omega_E injective}
Let $E$, $A$ be Abelian groups. The bilinear map:
\begin{equation}\label{NEEW}
\omega: E\times Hom(E,A)\longrightarrow A,\quad (x,g)\mapsto g(x)
\end{equation}
is separated if and only if $\omega_E$ is injective. In particular, $\mathbb{H}(E,Hom(E,A),\omega)$ is a generalized Heisenberg group if and only if $\omega_E$ is injective.
\end{lemma}

\begin{proof} It follows from Remark \ref{omegaE}, since $\omega_{Hom(E,A)}=id_{Hom(E,A)}$ is injective. \end{proof}

\begin{definition}\label{standard H} Let $E$, $A$ be Abelian groups such that the bilinear map
\begin{displaymath}
\omega: E\times Hom(E,A)\longrightarrow A,\quad (x,f)\mapsto f(x)
\end{displaymath} is separated. Then $\HG$ is called {\em standard Heisenberg group} and denoted by $\mathbb{H}(E,A)$.
\end{definition}

This family is very relevant since any generalized Heisenberg group can be embedded in a standard one.

\begin{lemma}\label{gen embed in standard}
Every generalized Heisenberg group embeds into a standard Heisenberg group. 
\end{lemma}

\begin{proof} Let $G=\mathbb{H}(E,F,A,\omega)$ be a generalized Heisenberg group. By Remark \ref{omegaE}(b), $\omega_F$ is an embedding of $F$ in $Hom(E,A)$. Therefore, the map 
\begin{equation}\label{JE}
\iota_E= id_E\times \omega_F\times id_A: \HG\longrightarrow \mathbb{H}(E,A)
\end{equation}
is an injective homomorphism providing the required embedding in $\mathbb{H}(E,A)$. Similarly, one can embed $G$ in $\mathbb{H}(F,A)$ using the flip $\tau: Hom(F,A)\times F\longrightarrow F\times Hom(F,A)$ defined by $(f,x)\mapsto (x,f)$, for $f\in Hom(F,A)$ and $x\in F$. The map
\begin{equation}\label{JF}
\iota_F:=(\tau\times id_A)\circ(\omega_E\times id_F\times id_A):\HG\longrightarrow \mathbb{H}(F,A)
\end{equation}
is an injective homomorphism providing the required embedding.
\end{proof}


\section{Topological Heisenberg groups}
In this section we show that the Mumford map has additional nice properties in the case of topological groups. 
We endow generalized Heisenberg groups with a topology making them topological groups and study 
their properties. In particular, we characterize locally compact generalized Heisenberg group with center isomorphic to $\mathbb{T}$.

\begin{definition}
Let $G,H$ be topological groups. Then $Chom(G,H)$ is the group of the continuous homomorphism from $G$ to $H$ endowed with the compact-open topology. If $H=\mathbb{T}$ we denote $Chom(G,\mathbb{T})$ simply by $\widehat{G}$.
\end{definition}


First we adapt the Mumford map for the framework of topological groups.

\begin{remark}\label{topological commutator}
Let $G\in \mathcal{N}_2$ be a topological group and $\pi: G\longrightarrow G/Z(G)$ be the canonical map.

(a) As the commutator map is continuous, for every open neighborhood of the identity $U\subseteq G$, there exist $V,W$ open neighborhood of the identity such that $[V,W]\subseteq U$, so $[\pi(V),\pi(W)]=[V,W]\subseteq U$. Since $\pi$ is open, therefore the factorized commutator map $B$ is continuous

(b) It follows from item (a) that the Mumford map $M:K\longrightarrow Hom(K,Z(G))$ takes values in the subgroup $ Chom(K,Z(G))$ of $Hom(K,Z(G))$ consisting of  continuous homomorphisms from $K$ to $Z(G)$, where $K=G/Z(G)$ carries the quotient topology and $Z(G)$ is endowed with the induced by $G$ topology.
\end{remark}

In order to emphasize the fact described in item (b) of the above remark, we adapt the definition of the Mumford map as follows:

\begin{definition} Let $G\in \mathcal{N_2}$ be a topological group and $K=G/Z(G)$. We use the notation $\Mt$ for the corestriction
\begin{displaymath}
\Mt:K\longrightarrow Chom(K,Z(G)),
\end{displaymath}
defined as in formula \eqref{mumford map formula}, and we call it {\it topological Mumford map}.
\end{definition}

Now we see that the topological Mumford map $\Mt$ is a continuous injective homomorphism whenever $G\in \mathcal{N}_2$ is a topological group.
%
%
%
%
%
%

\begin{lemma}\label{Btilde is injective and continuous}
Let $G\in \mathcal{N}_2$ be a topological group and $K=G/Z(G)$. Then the topological Mumford map $\Mt$ is a continuous injective homomorphism.  
\end{lemma}

\begin{proof} Injectivity follows from Lemma \ref{commutator map}. An open neighborhood of the identity of $Chom(K,Z(G))$ is given by 
\begin{displaymath}
W_{L,U}=\lbrace \chi: K\longrightarrow Z(G),\quad \chi (L)\subset U\rbrace
\end{displaymath} 
for some compact $L\subseteq K$ and some open neighborhood of the identity $U\subseteq Z(G)$.\\ 
The map $B$ is continuous by Remark \ref{topological commutator}, so for every $y\in L$ there exists a open neighborhood $V_y$ 
of 0 such that and $[V_y,y+V_y]\subseteq U$. The family $\lbrace y+V_y, y\in L\rbrace$ is a open cover of $L$, so there exists a finite subcover $\lbrace y_i+V_{y_i}, 1\leq i\leq r\rbrace$. Then $V=\bigcap_{i=1}^r V_{y_i}$ is open and $\Mt(V)\subseteq W_{L,U}$. Therefore, $\Mt$ is continuous. 
 \end{proof}
 
 In case $A, E$ and $F$ are topological groups, one can provide  $(\HG,\gamma)$ with the product
 topology $\gamma$ which makes it a topological group provided $\omega$ is continuous. 
Actually, one has the following more precise result: 

\begin{proposition}\label{Prop:Misha}{\cite{Me1}} 
Let $A,E,F$ be Abelian groups and $\nu,\sigma,\tau$ be topologies on $A, E$ and $F$, respectively. 
Then the product topology $\gamma$  on $E\times F\times A$ is a group topology of $(\HG,\gamma)$ if and only if $(A,\nu),(E,\sigma),(F,\tau)$ are topological groups and $\omega$ is continuous.
\end{proposition}

\begin{definition}
Let $E,A$ be topological Abelian groups and $Chom(E,A)$ be the group of continuous homomorphism from $E$ to $A$ endowed with the compact open topology. If the map
\begin{equation}\label{ev map}
\omega:E\times Chom(E,A)\to A, \ (x,f)\mapsto f(x)
\end{equation}
is separated and continuous, then $\mathbb{H}(E,Chom(E,A),A,\omega)$ endowed with the product topology is a topological group that will be called
\emph{topological standard Heisenberg group} and denoted by $\mathbb{H}(E,A)$.
\end{definition}

\begin{example}\label{Mackey}
(a) Let $E$ be a LCA group. The map
\begin{equation}\label{NEEEW}
\omega: E\times \widehat{E}\longrightarrow \mathbb{T},\hspace{1 cm} (x,f)\mapsto f(x)
\end{equation} 
is a continuous separated bilinear map, so $G=\mathbb{H}(E,\widehat{E},\omega)$ is a topological group. As mentioned in the introduction, 
it will be denoted by $\mathbb{H}(E)$ and called {\em Mackey-Weil group}. The isomorphisms $\mathbb{H}(E)\simeq (E\times \mathbb{T})\rtimes \widehat{E}$,  
$\mathbb{H}(\widehat{E})\simeq (\widehat{\widehat{E}} \times \mathbb{T})\rtimes \widehat{E}$ given by Lemma \ref{heisenberg as semidirect}, and $E\simeq \widehat{\widehat{E}}$ yield $\mathbb{H}(E)\simeq \mathbb{H}(\widehat{E})$.\end{example}

(b) If $E$ is a topological abelian group, then (\ref{NEEEW}) need not be continuous. An example to this effect is the group $E = \mathbb Z^{\#}$, namely, the group $\mathbb Z$ equipped with its Bohr topology. 
Then $ \widehat{E} = \mathbb T$ equipped with the usual compact topology of $\mathbb T$. 
 
\begin{corollary}\label{evaluetion are continuous} Let $\HG=\mathbb{H}(E,F,A,\omega)$ be a topological generalized Heisenberg group. Then
\begin{equation}\label{Laaast:Eq}
\omega_E: E \longrightarrow  Chom(F,A)\quad \text{and} \quad \omega_F: F \longrightarrow  Chom(E,A)
\end{equation}
are continuous injective homomorphisms.
\end{corollary}

\begin{proof} The map $\Mt:E\times F\longrightarrow Chom(E\times F,A)\simeq Chom(E,A)\times Chom(F,A)$ is given by 
\begin{displaymath}
\Mt(x,y)(z,u)=\omega_E(x)(u)-\omega_F(y)(z)
\end{displaymath}
for every $x,z\in E$ and $y,u\in F$ and it is a continuous injective group homomorphism by Lemma \ref{Btilde is injective and continuous}. Therefore,
the maps in (\ref{Laaast:Eq}) are continuous injective group homomorphisms. 
%
 \end{proof}

\begin{remark}  
For the standard Heisenberg group $G=\mathbb{H}(E,A)$, the following conditions are equivalent:  
\begin{itemize}

\item[(a)] $G$ with the product topology is a topological group; 

\item[(a$'$)] $\omega: E\times Chom(E,A)\longrightarrow A$ is continuous;

\end{itemize}
For abelian topological groups $E$ and $A$ and $F = Chom(E,A)$ endowed with the compact-open topology, the evaluation map $\omega_{E}: E \to Chom(Chom(E,A),A)$ is continuous if and only if the compact subsets of $Chom(E,A)$ are equicontinuous (this result is well-known for $A=\mathbb{T}$).
By Corollary \ref{evaluetion are continuous}, condition (a) implies the following equivalent conditions:
\begin{itemize}
\item[(b)] (\ref{Laaast:Eq}) are continuous; 

\item[(c)] the evaluation map $\omega_{E}$ is continuous;

\item[(d)] the compact subsets of $Chom(E,A)$ are equicontinuous. 
\end{itemize}
So conditions (b)-(c)-(d) are necessary for $G$ to be a topological group. 
\end{remark}

 The next example shows that none of the continuous injections $\omega_E$ and $\omega_F$ in Corollary \ref{evaluetion are continuous} are not topological embeddings in general.

\begin{example} Let $E$ and $F$ be locally compact abelian groups such that there exists a dense injective non-surjective continuous homomorphism $\iota: E \to \widehat F$. Then the dual map $\widehat \iota$ composed with the isomorphism $F\to \widehat{\widehat F}$ provides a dense injective non-surjective  continuous homomorphism $\iota': F \to \widehat E$. Clearly, none of $\iota, \iota'$ is a topological embedding as locally compact groups are complete. 

Define $\omega: E \times F\to \mathbb{T}$ by $\omega(x,y) =\iota(x)(y)$ for all $x,y\in E \times F$.  Then, for the topological generalized Heisenberg group $\mathbb{H}(E,F, \mathbb{T},\omega)$ one can identify $\omega_E$ with $\iota$ and $\omega_F$ with $\iota'$. Hence, both $\omega_E$ and $\omega_F$ fail to be topological embeddings.  Note that the topological generalized Heisenberg group $\mathbb{H}(\widehat E, \widehat F, \mathbb{T},\omega)$ has the same property, although it need not be isomorphic to $\mathbb{H}(E,F, \mathbb{T},\omega)$. 

For an example of pair $E,F$ with these properties take the discrete group $E = \mathbb{Q}$, $F = \mathbb{R}$ and $\iota: \mathbb{Q} \to \mathbb{R} =E$ the inclusion map. Note the groups $\mathbb{H}(E, F, \mathbb{T},\omega)$ and $\mathbb{H}(\widehat E, \widehat F, \mathbb{T},\omega)$ are not even homeomorphic, as the latter group is connected, while the former one is not.
\end{example}

The next Theorem offers a characterization of the locally compact generalized Heisenberg groups with $Z(G) = \mathbb{T}$ in the class $\mathcal N_2$.  For the proof we need the following well known fact: 
 
 \begin{fact}{\cite[6.16]{A}}\label{T is direct summand}
Let $G$ be a locally compact abelian group and $\mathbb{T}\leq G$. Then $\mathbb{T}$ is a direct summand of $G$.\end{fact}
\begin{theorem}\label{gen iff sss*} For a  locally compact group $G\in \mathcal N_2$ with $Z(G) \cong \mathbb{T}$ the following are equivalent:
\begin{enumerate}
  \item $G$ is a generalized Heisenberg group;
  \item $G/Z(G)\simeq E\times F$, where $E$ and $F$ are maximal isotropic subgroups with respect to $B$.
\end{enumerate}
\end{theorem}

\begin{proof} (1) $\Rightarrow$ (2) By Theorem \ref{caratt generalized}, there exists   a pair of maximal abelian subgroups $M_1,M_2$ of $G$ with $M_1\cap M_2=Z(G)$ and $M_1\cdot M_2= G$. By Lemma \ref{isotropic and abelian}, this gives rise to a pair of maximal isotropic subgroups $E=\pi(M_1)$ and $F=\pi(M_2)$ of $G/Z(G)$ such that $E\cap F=\{ 0\}$ and $E + F = G/Z(G)$. Therefore, $G/Z(G)\simeq E\times F$. 

(2) $\Rightarrow$ (1) Assume that $G/Z(G)\simeq E\times F$, where $E$ and $F$ are maximal isotropic subgroups with respect to $B$. 
Hence, $E\cap F=\{ 0\}$ and $E + F = G/Z(G)$. By Lemma \ref{isotropic and abelian}, this gives rise to a pair of maximal  abelian subgroups $M_1,M_2$ of $G$ with $M_1\cap M_2=Z(G)$ and $M_1\cdot M_2= G$. Moreover,  $Z(G)\cong \mathbb{T}$ is a direct summand of $M_1$ and $M_2$, by Fact \ref{T is direct summand}. Then $G$ satisfies conditions of Theorem \ref{caratt generalized}, so it is a  generalized Heisenberg group. 
\end{proof}

\section{Mumford Groups}

In this section we present another class of topological groups of $\mathcal{N}_2$, named {\em Weyl-Mumford groups}, since inspired by a notion introduced by Mumford in \cite{Mu}, in the framework of topological groups under more stringent conditions. The topological groups obtained in this way 
cover the special case treated by Prasad and Vemuri in \cite{Prasad}, where they carried out this 
construction only for {\em locally compact abelian groups $G$ with $Z(G) =\mathbb{T}$. }

\begin{definition}\label{M-groups} A {\it Mumford group} is a topological group $(G,\tau)\in \mathcal{N}_2$, such that the topological Mumford map of $G$,
\begin{displaymath}
\Mt:G/Z(G)\longrightarrow Chom(G/Z(G), Z(G))
\end{displaymath}
 is a topological isomorphism. 
\end{definition}

\begin{lemma}\label{def_psi}
Let $E,F$ and $A$ be locally compact abelian (LCA) groups. Then map
\begin{displaymath}
\psi: Chom(E,A)\times Chom(F,A)\longrightarrow Chom(E\times F,A)
\end{displaymath} defined by setting $\psi(f,g)(x,y)=g(y)-f(x)$ for every $f\in Chom(E,A)$, $g\in Chom(F,A)$, $x\in E$ and $y\in F$ is a topological isomorphism.
\end{lemma}

\begin{proof}  Let $U\subseteq A$ be an open subset, $L\subseteq E\times F$ be compact and let $\pi_E$ and $\pi_F$ the projections onto $E$ and $F$. Then $\pi_E(L)=L_E$ and $\pi_F(L)=L_F$ are compact and there exists open subset $V$ of $A$ such that $V-V\subseteq U$. Therefore, $\psi(f,g)\in W_{L,U}$ for every $(f,g)\in W_{L_E,V}\times W_{L_F,V}$, so the map $\psi$ is continuous.\\
The map given by $\phi(f)=(-f|_E,f|_F)$ for every $f\in Chom(E\times F,A)$, where $f|_E$ and $f|_F$ denotes respectively the restrictions of $f$ to $E$ and $F$ is the inverse of $\psi$. Since the restrictions are continuous, $\phi$ is continuous too.
\end{proof}

\begin{definition}\label{riflexive}
For the Abelian topological groups $E$ and $A$ and the bilinear map
\begin{displaymath}
\omega: E\times Chom(E,A)\longrightarrow A,\quad (x,g)\mapsto g(x)
\end{displaymath}
we call the group $E$ $A$-\textit{reflexive} if the map $\omega_E: E \to  Chom(Chom(E,A),A) $ is a topological isomorphism. 
\end{definition}

\begin{theorem}\label{Topgen are PVH iff}
Let $G=\mathbb{H}(E,F,A,\omega)$ be a topological generalized Heisenberg group. Then $G$ is a Mumford group if and only if $\omega_E$ and $\omega_F$ in  \eqref{Laaast:Eq} are topological isomorphisms, so $E$ and $F$ are $A$-reflexive.
\end{theorem}

\begin{proof} The following diagram, where $\psi$ is the isomorphism as in Lemma \ref{def_psi} and $\tau$ is the canonical flip, is commutative.
\begin{center}
$\xymatrixcolsep{8pc}\xymatrix{
E\times F\ar[r]^{\Mt}  \ar[d]^{\tau} & Chom(E\times F,A)  \\
F\times E \ar[r]^{\omega_F\times\omega_E}& Chom(E,A)\times Chom(F,A) \ar[u]^{\psi} .
}
$
\end{center} 
Since $\psi$ is a topological isomorphism, we can conclude that the topological Mumford map $\Mt$ is a topological isomorphism if and only if $\omega_F\times\omega_E$  is a topological isomorphism, hence if and only if both $\omega_E$ and $\omega_F$ are topological isomorphisms. \\
The map:
$$ 
E\longrightarrow Chom(Chom(E,A),A), \quad x\mapsto \omega_E(x)\circ \omega_F^{-1}
$$
is a topological isomorphism, and $\omega_E(x)(\omega_F^{-1}(f))=\omega(x,\omega_F^{-1}(f))=\omega_F(\omega_F^{-1}(f))(x)=f(x)$ for every $f\in Chom(E,A)$. Therefore, $E$ is $A$-reflexive. The same argument shows that $F$ is $A$-reflexive.
\end{proof}

Now we see the impact of imposing reflexivity in the description of the class of the topological Heisenberg groups. 

\begin{corollary}\label{standard and reflex}
A standard topological Heisenberg group $\mathbb{H}(E,A)$ is a Mumford group if and only if $E$ is $A$-reflexive. \end{corollary}

\begin{proof}
Follows from Theorem \ref{Topgen are PVH iff}, since $\omega_E$ has to be a topological isomorphism.
\end{proof}

It turns out that topological generalized Heisenberg group which are also Mumford are precisely the standard Heisenberg groups given by pair of groups $E$ and $A$ with $E$ $A$-reflexive.

\begin{corollary}\label{standard are PVH}
Let $G=\mathbb{H}(E,F,A,\omega)$ be a topological generalized Heisenberg group. Then the following are equivalent:
\begin{enumerate}
\item[(i)] $G$ is a Mumford group;
\item[(ii)] the maps $\iota_E$ and $\iota_F$ as defined in (\ref{JE}) and (\ref{JF}) are isomorphisms;
\item[(iii)] $E$ is $A$-reflexive;
\item[(iv)] $F$ is $A$-reflexive.
\end{enumerate}
\end{corollary}

\begin{proof}
(i) $\Leftrightarrow$ (ii) The group $G$ embeds in the standard Heisenberg groups $\mathbb{H}(E,A)$ and $\mathbb{H}(F,A)$, by Lemma \ref{gen embed in standard}. By Theorem \ref{Topgen are PVH iff}, $G$ is Mumford if and only if the embeddings defined in Lemma \ref{gen embed in standard} are both topological isomorphisms.\\
(ii) $\Leftrightarrow$ (iii) Follows by Corollary \ref{standard and reflex}, by virtue of the equivalence between (i) and (ii).\\ 
(ii) $\Leftrightarrow$ (iv) The same argument of (ii) $\Leftrightarrow$ (iii) applies.
\end{proof}
\begin{example}\label{non-ex}
\begin{itemize}
\item[(a)] Let $K$ be a discrete infinite field and let $G=\mathbb{H}(K,K,\mathbb{T},\omega)$ with $\omega(x,y)=xy$. Then $G$ is a locally compact generalized Heisenberg group with $Z(G)\cong \mathbb{T}$. The map $\omega_K:K\longrightarrow \widehat{K}$ is not a topological isomorphism, since group $\widehat{K}$ is compact and $K$ is infinite. By Corollary \ref{standard are PVH}(iii), $G$ is not a Mumford group.

\item[(b)] Let $E$ and $A$ be topological groups, such that $Chom(E,A)$ separates the points of $E$
and its compact-open topology is not discrete (e.g., take any locally compact, non compact group $E$ and $A=\mathbb{T}$).
Now let $F= Chom(E,A)$, endowed with the discrete topology. Then $G=\mathbb{H}(E,F,A,\omega)$, with $$\omega:E\times Chom(E,A)\to A,\quad (x,f)\mapsto f(x),$$ is a topological generalized Heisenberg group. Since $\omega_F=id_{Chom(E,F)}$ is not open, $G$ is not a Mumford group by Corollary \ref{standard are PVH}(ii).
\end{itemize}
\end{example}

A family of groups which are both generalized Heisenberg and topological Mumford are given by Mackey-Weil groups. For them the map $\omega_{G}$ is the natural transformation arising from Pontryagin duality.

\begin{corollary}\label{H(G) is WM} Every Mackey-Weil group $\mathbb{H}(E)$ is a Mumford group. \end{corollary}

\begin{proof}
It follows from Corollary \ref{standard are PVH}, since the map $\omega_E$ defined by setting $\omega_E(x)(\chi)=\chi(x)$, for every $x\in E$ and every $\chi\in \widehat{E}$ is a topological isomorphism, since $E$ is a LCA group.
\end{proof}
 
\begin{example}\label{examples and non-examples}  
\begin{itemize}
\item[(a)] Consider $\mathbb{R}$ endowed with the usual topology. The group $\mathbb{H}(\mathbb{R})$ is both a generalized Heisenberg group and a Mumford group, since it is a Mackey-Weil group.\\ The Heisenberg group $\mathbb{H}_{\mathbb{R}}=\mathbb{H}(\mathbb{R},\mathbb{R},\mathbb{R},\omega)$, with $\omega(x,y)=xy$ is a Mumford group, since $$\omega_{\mathbb{R}}:\mathbb{R}\to Chom(\mathbb{R},\mathbb{R}), \quad x\mapsto (y\mapsto xy)$$ is a topological isomorphism.
\item[(b)] Let $K$ be a topological field, such that $\mathbb{Q}$ is a dense subgroup of $K$. The topological Heisenberg group $\mathbb{H}_K=\mathbb{H}(K,K,K,\omega)$, with $\omega(x,y)=xy$ is a Mumford group, since $$\omega_{K}:K\to Chom(K,K), \quad x\mapsto (y\mapsto xy)$$ is a topological isomorphism.
\end{itemize}
\end{example}


\section{Locally Compact Mumford groups and Symplectic Self-Dualities}

In this section we discuss symplectic self-dualities and their relation to locally compact Mumford groups $G$ with $Z(G)=\mathbb{T}$.

\begin{definition} For a LCA group $K$, a {\em symplectic self-duality} is an isomorphism 
\begin{displaymath}
\nabla:K\longrightarrow \widehat{K}
\end{displaymath} such that $\nabla(x)(x)=0$ for every $x\in K$. 
\end{definition}

Every self-duality $(K,\nabla)$ defines a bilinear form $b_\nabla: K\times K \longrightarrow \mathbb{T}$, by letting $b_\nabla(x,y) = \nabla(x)(y)$
for $x,y \in K$. The bilinear form $b_\nabla$ is separated, as $\widehat{K}$ separates the points of $K$.

Two self-dualities $(K,\nabla)$ and $(K',\nabla')$ are said to be  {\em  isomorphic} if the corresponding bilinear forms $b_\nabla$
and $b_{\nabla'}$ are isomorphic (i.e.,  there exists an isomorphism $\phi:K\longrightarrow K'$ such that $\nabla'(\phi(x))(\phi(y))=\nabla(x)(y)$ for every $x,y\in K$). Clearly, a self-duality $(K,\nabla)$ is symplectic precisely when its bilinear form $b_\nabla$ is alternating. 

A subgroup $H$ of $K$ is called {\em isotropic w.r.t. $\nabla$}, if $H$ is isotropic w.r.t. $b_\nabla$, i.e., if $\nabla(x)(y)=0$ for every $x,y\in H$, or, shortly,  $\nabla(H)(H)=0$).

\begin{example}\label{Exa:ssd} Let $A$ be a LCA group. Then $K=A\times \widehat{A}$ is a LCA group and  the map $\nabla_A: K \to \widehat K$, defined by $\nabla_A(x,f)(y,g)=g(x)-f(y)$, is a  symplectic self-duality.  Following \cite{Prasad}  we call a symplectic self-duality {\em standard } if it is 
isomorphic to $(A\times \widehat{A}, \nabla_A)$ for some locally compact abelian group $A$. 
\end{example}

A suitable characterization of standard self dualities can be found as follows:  

\begin{lemma}\label{maximal iso for standard}
Let $A$ be a LCA group and $\nabla$ be the standard self duality related to $K=A\times \widehat{A}$. Then $A$ and $\widehat{A}$ are maximal isotropic subgroups of $K$.
\end{lemma}

\begin{proof}
 The subgroups $A\times\lbrace 0\rbrace$ and $\lbrace 0\rbrace\times \widehat{A}$ are isotropic, since $\nabla(x,0)(y,0)=0$ and $\nabla(0,f)(0,g)=g(0)-f(0)=0$. Let us assume that $A\times\lbrace 0\rbrace\leq B$ for some isotropic subgroup $B$. If $(y,g)\in B$, then 
\begin{eqnarray*}
\nabla(x,0)(y,g)= g(x)=0
\end{eqnarray*}
for every $x\in A$. Hence, $g=0$ and $B= A\times\lbrace 0\rbrace$. Similalry, let $ \{0\}\times \widehat{A}\leq B$ be a isotropic subgroup and $(x,g)\in B$. Then $\nabla(0,h)(x,g)=-h(x)=0$ for every $h\in \widehat{A}$. Since $\widehat{A}$ separates the point, then $x=0$.
\end{proof}

This property characterizes standard self dualities up to isomorphism as shown in \cite[Lemma 4.2]{Prasad}. Namely,  
if $(K,\nabla)$ is a symplectic self duality such that $K\simeq E\times F$ with $E$ and $F$ maximal isotropic subgroup, then $(K,\nabla)$ is a standard self-duality. 

\begin{remark}\label{PVH>ssd} It turns out that the Mumford groups $G$ such that $Z(G)=\mathbb{T}$ supply symplectic self-dualities. Indeed, 
if $G$ is such a group, then the topological Mumford map $\Mt: K \to \widehat K$ is a symplectic self-duality on $K=G/Z(G)$.
\end{remark}

The next theorem offers a characterization of generalized Heisenberg groups in the family of Mumford groups, with 
center  topologically isomorphic to $\mathbb T$, via properties of the symplectic self-duality $(K,\Mt)$. It turns out that they are Mackey-Weil groups up to isomorphism. 

\begin{theorem}\label{gen iff sss}
Let $G$ be a locally compact Mumford group such that $Z(G)=\mathbb{T}$ and let $K=G/Z(G)$. The following are equivalent:
\begin{enumerate}
 \item $G$ is a generalized Heisenberg group;
 \item $(K,\Mt)$ is a standard self-duality;
 \item $G$ is isomorphic to a Mackey-Weil group $\mathbb{H}(E)$ for some LCA group $E$.
\end{enumerate}
\end{theorem}
\begin{proof} 
%
(1) $\Leftrightarrow$ (3) By Corollary \ref{standard are PVH}, we have that a generalized Heisenberg group $G$ is Mumford if and only if it is isomorphic to some standard Heisenberg group $\mathbb{H}(E,A)$. In this case $A=\mathbb{T}$ and $E$ is LCA, so $G$ is isomorphic to the Mackey-Weil group $\mathbb{H}(E)$.

(1) $\Rightarrow$ (2) Let $K=G/Z(G)$. According to Theorem \ref{heisenberg as semidirect}, $G\simeq \mathbb{H}(E,F,Z(G),\omega)$ where $\omega:E\times F\longrightarrow Z(G)$ defined by setting $\omega(x,y)=[i(x),j(y)]$ where $i$ and $j$ are embeddings of $E$ and $F$ in $G$ and the isomorphism is given by $\phi(x,y,z)=i(x)j(y)z$, for every $x \in E$, $y\in F$ and $z\in Z(G)$. Therefore $\phi^\prime:E\times F\longrightarrow K=G/Z(G)$, given by $(x,f)\longrightarrow i(x)j(y)Z(G)$ is an isomorphism. By Theorem \ref{Topgen are PVH iff}, $\omega_F$ is a topological isomorphism, as $G$ is a Mumford group. Then the map $\psi=\phi^\prime\circ (id_E\times \omega_F^{-1}): E\times \widehat{E}\longrightarrow K$ is an isomorphism and $\phi ([(x,z,0),(y,u,0)])=[\phi(x,z,0),\phi(y,u,0)]=B(\phi^\prime(x,z),\phi^\prime(y,u))$ for every $x,y\in E$, $z,u\in F$. So we have:
\begin{eqnarray*}
\Mt(\psi(x,f))(\psi(y,g))&=&B(\phi^\prime(x,\omega_F^{-1}(f),\phi^\prime(y,\omega_F^{-1}(g))=\\
&=&\phi ([(x,\omega_F^{-1}(f),0),(y,\omega_F^{-1}(g),0)])=\\
&=& \omega(x,\omega_F^{-1}(g))-\omega(y,\omega_F^{-1}(f))=\\
&=& g(x)-f(y).
\end{eqnarray*}
Therefore, $(K,\Mt)$ is a standard self duality.

(2) $\Rightarrow$ (1) Let $G$ be a locally compact Mumford group and let us assume that $(K,\Mt)$ is a standard self-duality. So, there exist a LCA $A$ and an isomorphism $\phi: A\times \widehat{A}\longrightarrow K$. By Lemma \ref{maximal iso for standard}, $E=\phi(A)$ and $F=\phi(\widehat{A})$ are maximal isotropic subgroups of $K$ and $K\simeq E\times F$, since $(K,\Mt)$ and $(A\times\widehat{A},\nabla)$ are isomorphic. According to Theorem \ref{gen iff sss*}, $G$ is a generalized Heisenberg group. 
\end{proof}

The equivalence between (1) and (2) can be alternatively obtained by using the characterization of the standard self-dualities given in Lemma 4.2 of \cite{Prasad}.

\section{Final comments and open questions}

In Remark \ref{PVH>ssd} we pointed out that every locally compact Mumford group $G$ with $Z(G)=\mathbb{T}$
gives rise to a symplectic self-duality on $K=G/Z(G)$ via its topological Mumford map $\Mt$. Moreover, two different groups correspond to the same $\Mt$ if and only if they are isomorphic (\cite[Proposition 3.5]{Prasad}). 

On the other hand, if $(K,\nabla)$ is a standard symplectic self duality one can obviously
construct a Mumford (actually a generalized Heisenberg) group for which $\Mt=\nabla$.
It was proved in \cite{Vemuri} that this is always the case when $K$ is finite:

\begin{theorem}\label{K finite}\cite{Vemuri}
Let $K$ be a finite abelian group and $B:K\times K\longrightarrow \mathbb{T}$ a separated alternating bilinear map. Then there exist a finite abelian group $A$ and an isomorphism $\phi: A\times \widehat{A}\longrightarrow K$ and such that $B(\phi(x,f),\phi(y,g))=g(x)-f(y)$.
\end{theorem}

\begin{corollary}
Let $G\in\mathcal{N}_2$ be a locally compact group such that $Z(G)\simeq \mathbb{T}$ and $K=G/Z(G)$ is finite. Then $G$ is a Mumford group and a generalized Heisenberg group.
\end{corollary}

\begin{proof}
The topological Mumford map $\Mt$ is injective. Since $K$ is finite then $K\simeq \widehat{K}$ and then $\Mt$ is also surjective.\\
Putting together Theorem \ref{K finite} and Theorem \ref{gen iff sss} it follows that $G$ is generalized Heisenberg group.
\end{proof}

Now we discuss the possibility to invert Remark \ref{PVH>ssd}:  

\begin{question}\label{PVH da ssd} 
Let $(K,\nabla)$ be a symplectic self duality. Does there exist a locally compact group $G\in \mathcal{N}_2$ with $Z(G)\simeq \mathbb T$, $G/Z(G) \cong K$ and $\Mt=\nabla$? 
\end{question}

\begin{remark}\label{last:Rem} A group $G\in \mathcal{N}_2$ is a central extension of two Abelian groups, and therefore its structure is determined by a cocycle $\gamma:G/Z(G)\times G/Z(G)\to Z(G)$ (see Remark \ref{B and tilde for generalized}(b)). The relation between $\gamma$ and the Mumford map $\Mt$ of $G$ is the following: 
\begin{equation}\label{Mumford e cociclo}
\Mt(x)(y)=-\gamma(x,-x)-\gamma(y,-y)+\gamma(-x,-y)+\gamma(x,y)+\gamma(-x-y,x+y).
\end{equation}
Note that the expression in \eqref{Mumford e cociclo} is indipendent on the choice of the section defining $\gamma$, since the value of $\Mt(x)(y)$ depends just on the classes $xZ(G)$ and $yZ(G)$ (since $\Mt$ is defined as in \eqref{mumford map formula}).
\end{remark}
So, Question \ref{PVH da ssd} can be stated in the following way involving cohomology: 
for a given symplectic self duality $(K,\nabla)$ does there exist a continuous cocycle $\gamma: K\times K\longrightarrow \mathbb{T}$ satisfying 
\begin{equation}\label{nabla e cociclo}
\nabla(x)(y)=-\gamma(x,-x)-\gamma(y,-y)+\gamma(-x,-y)+\gamma(x,y)+\gamma(-x-y,x+y) \ ? 
\end{equation}
 
According to Theorem \ref{K finite}, the answer to Question \ref{PVH da ssd} for a finite group $K$ is positive and moreover the map $\gamma$ can be chosen to be bilinear.

We have seen in \ref{non-ex} that there exist topological generalized Heisenberg groups with center isomorphic to $\mathbb{T}$ which are not Mumford group. The following question is still open.
\begin{question}\label{Laaaaast}
Does there exist Mumford groups with $Z(G)=\mathbb{T}$ that are not generalized Heisenberg groups?\end{question}
A positive answer to \ref{PVH da ssd}, provides examples towards a positive solution to question \ref{Laaaaast}, making use of Theorem \ref{gen iff sss} as follows. 
According to \cite[Theorem 11.2]{Prasad}, there exists a symplectic duality $(K,\nabla)$ such that  $K\cong A\times \widetilde A$ holds for no 
locally compact abelian group $A$. Then a locally compact group $G\in \mathcal{N}_2$ with $Z(G)\cong\mathbb T$, $G/Z(G) \cong K$ and $\Mt=\nabla$ will certainly be a Mumford group that is not generalized Heisenberg group, by Theorem \ref{gen iff sss} (2). 
The existence of such a group is granted by the positive answer to \ref{PVH da ssd}.

\end{document}